\documentclass[a4paper, 11 pt]{article}

\usepackage{amsfonts,amsmath,amssymb,hhline,ifthen,amsthm,graphics,cite,latexsym, caption2}
\usepackage{latexsym}
\usepackage[T2A]{fontenc}
\usepackage[cp1251]{inputenc}
\usepackage[english]{babel}

\usepackage[matrix, arrow, curve]{xy}
\usepackage{graphicx}

\newcounter{q3}
\newcommand{\dsm}[3]{
{\if#20{\if#31{\frac{\partial #1}{\partial y}}\else
          {\frac{\partial^{#3} #1}{\partial y^{#3}}}
        \fi}\else
  {\if#30{\if#21{\frac{\partial #1}{\partial x}}\else
            {\frac{\partial^{#2} #1}{\partial x^{#2}}}
          \fi}\else
    {\setcounter{q3}{#2}\addtocounter{q3}{#3}
    \frac{\partial{\if{1}\arabic{q3}^{}\else^{ \arabic{q3} }\fi}#1}
    {{\if#20\else{\partial x{\if#21\else{^{#2}}\fi}}\fi}
     {\if#30\else{\partial y\if#31\else{^{#3}}\fi}\fi} }}
   \fi}
\fi} }

\newcommand{\dzm}[2]{
{\if#10{\if#21{\frac{\partial}{\partial\overline{\zeta}}}\else
          {\frac{\partial^{#2}}{\partial\overline{\zeta}^{#2}}}
        \fi}\else
  {\if#20{\if#11{\frac{\partial}{\partial\zeta}}\else
            {\frac{\partial^{#1}}{\partial\zeta^{#1}}}
          \fi}\else
    {\setcounter{q3}{#1}\addtocounter{q3}{#2}
    \frac{\partial{\if{1}\arabic{q3}^{}\else^{ \arabic{q3} }\fi}}
    {{\if#10\else{\partial\zeta{\if#21\else{^{#1}}\fi}}\fi}
     {\if#20\else{\partial\overline{\zeta}\if#21\else{^{#2}}\fi}\fi} }}
   \fi}
\fi} }

\newcounter{q2}

\newtheoremstyle{theor}
  {\medskipamount}
  {\medskipamount}
  {\itshape}
  {}
  {\bfseries}
  {.}
  {.5em}
  {}

\newtheorem{definition}{Definition}[section]
\newtheorem{theorem}[definition]{Theorem}
\newtheorem{lemma}[definition]{Lemma}
\newtheorem{proposition}[definition]{Proposition}
\newtheorem{corollary}[definition]{Corollary}

\theoremstyle{definition}
\newtheorem{remark}[definition]{Remark}
\newtheorem{example}[definition]{Example}

\numberwithin{equation}{section}

\makeindex

\newtheoremstyle{remarks}
  {0mm}
  {0mm}
  {\itshape}
  {}
  {\itshape}
  {.}
  {.5em}
  {}
\makeatletter \@addtoreset{equation}{section} \makeatother

\begin{document}

\subsection*{\center A KIND OF LOCAL STRONG ONE-SIDE POROSITY }\begin{center}\textbf{O. Dovgoshey and V. Bilet  } \end{center}
\parshape=5
1cm 13.5cm 1cm 13.5cm 1cm 13.5cm 1cm 13.5cm 1cm 13.5cm \noindent \small {\bf Abstract.}
  We define and study the completely strongly porous at 0 subsets of $[0, \infty).$ Several characterizations of these subsets are obtained, among them the description via an universal property and structural one.

\parshape=5
1cm 13.5cm 1cm 13.5cm 1cm 13.5cm 1cm 13.5cm 1cm 13.5cm \noindent \small {\bf Key words:} one-side porosity, local strong porosity.

\parshape=2
1cm 13.5cm 1cm 13.5cm  \noindent \small {\bf }

 \bigskip
\textbf{ AMS 2010 Subject Classification: } 28A10, 28A05


\large\section {Introduction } \hspace*{\parindent}
 The basic ideas concerning the notion of set porosity for the first time appeared in some early works of Denjoy \cite{D1},  \cite{D2} and Khintchine  \cite{Kh} and then were arisen independently in the study of cluster sets in 1967 (Dol\v{z}enko \cite{Dol1}). Denjoy was interested in obtaining a classification of perfect sets on the real line in terms of the relative sizes of the complementary intervals. Khintchine had required a convenient way of describing certain arguments that use density considerations. The notion of a set of $\sigma$-porosity was defined by E.~P.~Dol\v{z}enko  \cite{Dol1}. The basic structure of porous sets and $\sigma$-porous sets has been studied in \cite{H}, \cite{HV} and \cite{Tk}. A useful collection of facts related to the notion of porosity can be found in \cite{Th}. A number of theorems exists in the theory of cluster sets which use the notion of $\sigma$-porosity (see, for example, \cite{Y},\cite{Yo1}, \cite{Yo2}, \cite{Z}). No less important is a question about relationship between porosity and dimension. In many applications the information on the dimension of certain sets is obtained via porosity. See the use of porosity, for example, in connection with free boundaries \cite{KPS} and complex dynamics \cite{PR}. Estimates of dimension in terms of porosity were obtained for a wide variety of notions of porosity (and dimension) in \cite{BS}, \cite{EJ}, \cite{KR}, \cite{OV}, \cite{M}, \cite{TN}, \cite{TR}, etc.  The porosity (in an appropriate sense) of many natural sets and measures was investigated in \cite{BS}, \cite{VC}, \cite{KR}, \cite{MU}. Moreover, the relationship between porosity and other geometric concepts such as conical densities and singular integrals was explored in \cite{VC}, \cite{AV}, \cite{M}. Porosity is also a property which is preserved, for example, under quasisymmetric maps \cite{JV}. Thereby the notion of set porosity plays an implicit role in different questions of analysis.

Many nontrivial modifications of the notion of porosity are used at present. The comparison of different definitions, and a survey of results can be found in \cite{Z3}.  Our paper is also a contribution to this line of studies and we introduce a new subclass of strongly porous at 0 subsets of $\mathbb{R^{+}}=[0, +\infty).$  

Let us recall the definition of the right porosity.  Let $E$ be a subset of $\mathbb R^{+}.$
\begin{definition}\label{D1}
The right porosity of $E$ at 0 is the quantity
\begin{equation}\label{L1}
p^{+}(E,0):=\limsup_{h\to 0^{+}}\frac{\lambda(E,0,h)}{h}
\end{equation}
where $\lambda(E,0,h)$ is the length of the largest open subinterval of $(0,h)$ that
contains no point of $E$. The set $E$ is strongly porous on the right at 0 if $p^{+}(E,0)=1.$
\end{definition}
Let $\tilde \tau=\{\tau_n\}_{n\in\mathbb N}$ be a sequence of real numbers. We shall say
that $\tilde \tau$ is \emph{almost decreasing} if the inequality $\tau_{n+1}\le\tau_{n}$ holds
for sufficiently large $n.$ Write $\tilde E_{0}^{d}$ for the set of almost decreasing
sequences $\tilde \tau$ with $\mathop{\lim}\limits_{n\to\infty}\tau_{n}=0$ and having
$\tau_{n}\in E\setminus \{0\}$ for $n\in\mathbb N.$

We use the symbols $Ext E$ and $ac E$ to denote the exterior and, respectively, the set of all accumulation points  (relative to the space $\mathbb R^{+}$ with the standard topology) of a set $E\subseteq \mathbb R^{+}.$

\begin{remark}\label{rem*}
The set $\tilde E_{0}^{d}$ is empty if and only if $0\not \in ac E.$
\end{remark}

Define $\tilde I_{E}$ to be the set of sequences $\{(a_n, b_n)\}_{n\in\mathbb
N}$ of open intervals $(a_n,b_n)\subseteq
\mathbb R^{+}$ meeting the following conditions.

\bigskip
$\bullet$ \emph{Every $a_n$ is strictly positive.}

\bigskip
$\bullet$ \emph{Every interval $(a_n, b_n)$ is a connected component of $Ext E,$ i.e., $(a_n,b_n)~\cap~E=\varnothing$ but for every
$(a,b)\supseteq(a_n,b_n)$ we have $$((a,b)\ne (a_n, b_n))\Rightarrow((a,b)\cap E \ne
\varnothing).$$}

 \emph{$\bullet$ The limit relations
$\mathop{\lim}\limits_{n\to\infty}a_{n}=0$ and
$\mathop{\lim}\limits_{n\to\infty}\frac{b_n-a_n}{b_n}=1$ hold.}

\begin{remark}\label{1.3**}
$\tilde I_{E}\ne\varnothing$ if and only if $0\in ac E$ and $p^{+}(E,0)=1.$
\end{remark}

Define also an equivalence $\asymp$ on the set of sequences of strictly positive
numbers as follows. Let $\tilde a=\{a_n\}_{n\in\mathbb N}$ and
$\tilde{\gamma}=\{\gamma_n\}_{n\in\mathbb N}.$ Then $\tilde a \asymp \tilde {\gamma}$ if
there are  constants $c_1, c_2
>0$ such that
\begin{equation}\label{equiv1}
c_1 a_n \le \gamma_n \le c_2 a_n
\end{equation}
for sufficiently large  $n\in\mathbb N.$
\begin{definition}\label{D2*}
Let $E\subseteq\mathbb R^{+}$ and $\tilde \gamma
\in \tilde E_{0}^{d}.$ The set $E$ is $\tilde \gamma$-strongly porous if there is a
sequence $\{(a_n, b_n)\}_{n\in\mathbb N}\in\tilde I_{E}$ such that
\begin{equation}\label{equiv2}
\tilde\gamma \asymp \tilde a
\end{equation}
where $\tilde a=\{a_n\}_{n\in\mathbb N}.$ The set $E$ is completely strongly porous (at 0)
if $E$ is $\tilde \gamma$-strongly porous for every $\tilde \gamma \in \tilde
E_{0}^{d}.$
\end{definition}

\begin{remark}\label{1.3*}
If $0\not\in ac E,$ then $E$ is completely strongly porous because $\tilde E_{0}^{d}=\varnothing.$
\end{remark}

In what follows the set of all completely strongly porous subsets of $\mathbb R^{+}$ will be denoted by  $\textbf{\emph{CSP}}.$

The main results of the paper can be informally described by the following way.
\newline $\bullet$ $\textbf{\emph{CSP}}$ - sets are uniformly strongly porous (Theorem~\ref{ImpTh}), in the sense that the constants in \eqref{equiv1} can be chosen independently of $\tilde\gamma\in\tilde E_{0}^{d}$ if $E\in \textbf{\emph{CSP}}.$
\newline $\bullet$ If $E\in$ $\textbf{\emph{CSP}},$ then there is an universal $\tilde L \in \tilde I_{E}$ such that every $\tilde A\in\tilde I_{E}$ is a ``subsequence'' of $\tilde L$ (Theorem~\ref{ImpTh}).
\newline $\bullet$ A description of the structure of strongly porous on the right at 0 sets $E\subseteq\mathbb R^{+}$ having an universal $\tilde L\in\tilde I_{E}$ (Theorem~\ref{Th2.24}).
\newline $\bullet$ An explicit design generating all $\textbf{\emph{CSP}}$ - sets (Theorem~\ref{descr}).

 Olli Martio's question concerning interconnections between the infinitesimal structure of a metric space $(X,d)$ at a point $p\in X$ and the porosity of the distance set $\{d(x,p): x\in X\}$ was a starting point in our studies of $\textbf{\emph{CSP}}$ - sets. Some results in this direction can be found in \cite{ADK} and \cite{BD}.


\section{The \textbf{\emph{CSP}} - sets}
\hspace*{\parindent} We start in our investigations of the $\textbf{\emph{CSP}}$ - sets from the following two examples.
\begin{example}\label{ex1}
Let $\{x_n\}_{n\in\mathbb N}$ be a strictly decreasing sequence of positive real numbers
with $\mathop{\lim}\limits_{n\to\infty}\frac{x_{n+1}}{x_{n}}=0.$ Define a set $W$ by
the rule
$$(x\in W)\Leftrightarrow (\, \mbox{either} \,\, x=0 \,\,\mbox{or there is} \,\, n\in\mathbb N \,\,\mbox{such that} \,\,
x=x_{n}).$$ Then $W$ is a closed $\textbf{\emph{CSP}}$ - set and the sequence
$\{(x_{n+1}, x_{n})\}_{n\in\mathbb N}$ belongs to $\tilde I_{W}.$
\end{example}
\begin{example}\label{ex2}
Let $q\in [1, \infty)$ and let $W$ be the set from the previous example. Write
$$W(q)=\bigcup_{x\in W}[x, qx],$$ where $[x, qx]$ is the closed interval with the
endpoints $x$ and $qx.$ Then $W(q)$ is a closed $\textbf{\emph{CSP}}$ - set. Let
$m_{0}\in\mathbb N$ be a number such that $qx_{n+1}<x_{n}$ for every $n\ge m_{0}.$ The
sequence $\{(qx_{m_{0}+n+1}, x_{m_{0}+n})\}_{n\in\mathbb N}$ belongs to $\tilde I_{W(q)}.$
\end{example}
\begin{lemma}\label{Lem2.3}
Let $E\subseteq\mathbb R^{+},$
$\tilde\gamma\in\tilde E_{0}^{d},$ $\{(a_n, b_n)\}_{n\in\mathbb N}\in\tilde I_{E}$
and let $\tilde a:=\{a_n\}_{n\in\mathbb N}.$ The weak equivalence
$\tilde\gamma\asymp\tilde a$ holds if and only if
\begin{equation}\label{L2*}
\limsup_{n\to\infty}\frac{a_n}{\gamma_n}<\infty
\end{equation} and
\begin{equation}\label{L2**}
\gamma_n \le a_n
\end{equation} for sufficiently large $n.$
\end{lemma}
\begin{proof}
It is easily seen that \eqref{L2*} and \eqref{L2**} imply $\tilde \gamma\asymp\tilde a.$
Conversely suppose that $\tilde \gamma\asymp\tilde a.$ The membership $\{(a_n,
b_n)\}_{n\in\mathbb N}\in\tilde I_{E}$ yields $\frac{b_n}{a_n}\rightarrow\infty$
with $n\rightarrow\infty.$ Since $\tilde \gamma\asymp\tilde a,$ we
obtain $\frac{\gamma_n}{a_n}\le c_2$ for sufficiently large $n$ where $c_2$ is the constant from \eqref{equiv1}. Consequently there is
$N_0 \in \mathbb N$ such that the inequality \begin{equation}\label{L2***}\gamma_n<
b_n\end{equation} holds if $n\ge N_0.$ Since $(a_n,b_n)\cap E = \varnothing$ and
$\gamma_n \in E,$ inequality \eqref{L2***} implies \eqref{L2**}. To prove \eqref{L2*}
note that the left inequality in \eqref{equiv1} is equivalent to
$$\frac{1}{c_1}\ge\frac{a_n}{\gamma_n}$$ where $c_1$ is the constant from \eqref{equiv1}. Hence
\eqref{L2*} holds.
\end{proof}

\begin{corollary}\label{Col2.4}
Let $E\subseteq\mathbb R^{+}$
and let $\tilde\tau\in\tilde E_0^{d}.$ The set $E$ is
$\tilde\tau$-strongly porous if and only if there exists a
sequence $\{(a_n, b_n)\}_{n\in\mathbb N}\in\tilde I_{E}$ such that
$\mathop{\limsup}\limits_{n\to\infty}\frac{a_n}{\tau_n}<\infty$ and
$\tau_n\le a_n$ for sufficiently large $n.$
\end{corollary}

The following proposition does not have any applications in the paper but is used in \cite{BD} to describe the structure of bounded tangent spaces to general metric spaces.

\begin{proposition}\label{Pr5}
Let $E\subseteq\mathbb R^{+}$ and let
$\tilde \tau=\{\tau_n\}_{n\in\mathbb N}\in\tilde E_{0}^{d}.$ The following statements
are equivalent.
\begin{enumerate}
\item[\rm(i)]  $E$ is $\tilde\tau$-strongly porous.

\item[\rm(ii)] There is a constant $k\in (1, \infty)$ such that for every $K\in (k, \infty)$ there exists $N_{1}(K)\in \mathbb N$ such that
\begin{equation}\label{inters} (k\tau_n, K\tau_n)\cap E = \varnothing\end{equation} if $n
\ge N_1(K).$
\end{enumerate}
\end{proposition}
\begin{proof}
Suppose that $E$ is $\tilde\tau$-strongly porous. By Corollary~\ref{Col2.4} there
is a sequence
\begin{equation}\label{int}
\{(a_n, b_n)\}_{n\in\mathbb N}\in\tilde I_{E}
\end{equation} such that $\mathop{\limsup}\limits_{n\to\infty}\frac{a_n}{\tau_n}<\infty$
and $\tau_n \le a_n$ for sufficiently large $n.$ Write
$k=1+\mathop{\limsup}\limits_{n\to\infty}\frac{a_n}{\tau_n},$ then $\infty>k\ge 2$ and there is
$N_0 \in \mathbb N$ such that
\begin{equation}\label{L3*}
\tau_n\le a_n< k\tau_n
\end{equation} for $n\ge N_0.$
Let $K\in (k, \infty).$ Membership \eqref{int} implies the equality
$\mathop{\lim}\limits_{n\to\infty}\frac{b_n}{a_n}~=~\infty.$ The last equality and
\eqref{L3*} show that there is $N_1\ge N_0$ such that  $$a_n< k\tau_n < K\tau_n \le
b_n$$ if $n\ge N_1.$ Hence the inclusion \begin{equation}\label{L3**} (k\tau_n, K
\tau_n)\subseteq(a_n,b_n) \end{equation} holds if $n\ge N_1.$ Since
\begin{equation}\label{L3***} E\cap(a_n, b_n)=\varnothing,\end{equation}\eqref{L3**} implies \eqref{inters}. Thus (ii) follows from (i).

Conversely, assume that statement $\textrm{(ii)}$ holds. Let $K>1.$ Then for $K=2k$ there is $N_0\in\mathbb N$
such that $$(k\tau_n, 2k\tau_n)\cap E=\varnothing$$ if $n\ge N_0.$ Consequently, for
every $n\ge N_0,$ we can find a connected component $(a_n, b_n)$ of $Ext E$ meeting the
inclusion \begin{equation}\label{inql} (k\tau_n, 2k\tau_n)\subseteq (a_n, b_n).
\end{equation}
Write $(a_n, b_n)=(a_{N_0}, b_{N_0})$ for $n< N_0.$ Since, for $n\ge N_0,$ we have
$$\tau_n\in E,\, \tau_n < k\tau_n \, \, \mbox{and}\,\,(a_n, k\tau_n)\cap E=\varnothing,$$
the double inequality $\tau_n \le a_n< k\tau_n$ holds for such $n.$ To prove
$\textrm{(i)}$ it is sufficient to show that $$\{(a_n, b_n)\}_{n\in\mathbb N}\in \tilde
I_E.$$ All intervals $(a_n,b_n)$ are connected components of $Ext E$ and $\lim\limits_{n\to\infty}a_n = 0$ because $\lim\limits_{n\to\infty}\tau_n = 0,$ so that $\{(a_n, b_n)\}_{n\in\mathbb N}\in \tilde
I_E$ if and only if
\begin{equation}\label{infty} \lim_{n\to\infty}\frac{b_n}{a_n}=\infty.
\end{equation} Let $K$ be an arbitrary point of $(k, \infty).$ Applying \eqref{inters} we can find $N_1(K)\in\mathbb N$ such that $$(k\tau_n, K\tau_n)\subseteq (a_n, b_n)$$ for $n\ge N_1 (K).$
Consequently, for such $n,$ we have
$$\frac{b_n}{a_n}\ge\frac{K\tau_n}{k\tau_n}=\frac{K}{k}.$$ Letting $K\to\infty$ we see that \eqref{infty} follows.
\end{proof}
 It is clear that, if there is $\tilde\tau\in\tilde E_{0}^{d}$ such that $E$ is $\tilde\tau$-strongly porous, then $E$ is strongly porous on the right at 0. Conversely we have the following
\begin{proposition}\label{P1}
Let $E\subseteq\mathbb R^{+}$ and $0\in ac E.$
If $E$ is strongly porous on the right at 0, then there is $\tilde\tau\in\tilde
E_{0}^{d}$ for which $E$ is $\tilde\tau$-strongly porous.
\end{proposition}
The proof is immediate and can be omitted.
\begin{remark}
If $0\not\in ac E,$ then $E$ is strongly porous on the right at 0 but there are no $\tilde\tau \in \tilde E_{0}^{d}$ because $\tilde E_{0}^{d}=\varnothing.$
\end{remark}
\begin{definition}\label{D4}
Let $E \subseteq \mathbb R^{+}.$ The set $E$ is
uniformly strongly porous (at 0) if there exists a constant $c>0$ such that for every
$\tilde \tau\in \tilde E_0^{d}$ there is $\{(a_n,b_n)\}_{n\in\mathbb N}\in\tilde I_{E},$
satisfying the following conditions:
\item[\rm(i)]\textit{$a_n\ge\tau_n$ for sufficiently large $n\in\mathbb N;$}
\item[\rm(ii)]\textit{the inequality $$\limsup_{n\to\infty} \frac{a_n}{\tau_n}\le c$$ holds.}
\end{definition}

\begin{remark}\label{2}
If $0\not\in ac E,$ then $E$ is uniformly strongly porous since $\tilde E_{0}^{d}=\varnothing.$
\end{remark}

If $E$ is uniformly strongly porous, then $E\in \textbf{\emph{CSP}}.$
The converse is also true and we prove this in Theorem~\ref{ImpTh} giving
below.

Define, for $\tilde \tau \in \tilde E_0^{d},$ a subset $\tilde I_{E} (\tilde \tau)$ of
the set $\tilde I_E$ by the rule:
$$(\{(a_n ,b_n)\}_{n\in\mathbb N}\in\tilde I_{E}(\tilde \tau))\Leftrightarrow (\{(a_n ,b_n)\}_{n\in\mathbb N}\in\tilde I_{E}\, \mbox{and}\, \tau_n\le a_n \, \mbox{for sufficiently large}\, n\in\mathbb N).$$
Write
\begin{equation}\label{L9}
C(\tilde \tau):=\inf(\limsup_{n\to\infty}\frac{a_n}{\tau_n})\quad \mbox{and}\quad
C_E:=\sup_{\tilde \tau\in\tilde E_0^{d}}C(\tilde\tau)
\end{equation} where the infimum in the left formula is taken over all $\{(a_n ,b_n)\}_{n\in\mathbb N}\in\tilde I_{E}(\tilde \tau).$
\begin{proposition}\label{P2}
Let $E \subseteq \mathbb R^+$ and let $0\in ac E.$ The set $E$ is
strongly porous at 0 if and only if
\begin{equation}\label{L10}
\tilde I_{E}(\tilde \tau)\ne\varnothing
\end{equation}for every $\tilde \tau\in\tilde E_0^{d}.$ The set $E$ is completely strongly
porous if and only if $C(\tilde \tau)<\infty$ for every $\tilde\tau\in\tilde
E_0^{d}.$ The set $E$ is uniformly strongly porous if and only if $C_E<\infty.$
\end{proposition}
The proof follows directly from definitions \ref{D1}, \ref{D2*},
\ref{D4}, Corollary~\ref{Col2.4} and formulas~\eqref{L9}.
\begin{lemma}\label{Lem2.9}
Let $E\subseteq\mathbb R^{+}.$ If $\tilde
\tau=\{\tau_n\}_{n\in\mathbb N}\in\tilde E_{0}^{d} \quad \mbox{and} \quad \{(a_n, b_n)\}_{n\in\mathbb
N}\in\tilde I_{E}$ are sequences such that $\tilde a\asymp \tilde \tau,$ then $\tilde a$
and $\tilde b$ are almost decreasing.
\end{lemma}
\begin{proof}
 It suffices to show that $\tilde a$ is almost decreasing. If $\tilde a$ is not almost decreasing, then there is an infinite $A\subseteq\mathbb N$
such that \begin{equation}\label{X1} a_{n+1}>a_n
\end{equation} for every $n\in A.$ Since $(a_n, b_n)\cap E=\varnothing,$ inequality
\eqref{X1} implies that $a_{n+1}\ge b_{n}>a_n.$ By Lemma \ref{Lem2.3} we have $a_n
\ge\tau_n$ for sufficiently large $n.$ In addition, for such $n,$ we may suppose also
$\tau_n\ge\tau_{n+1}$ because $\tilde\tau$ is almost increasing. Consequently, we obtain
\begin{equation}\label{X2}a_{n+1}\ge b_n>a_n\ge\tau_{n}\ge\tau_{n+1} \end{equation} for
sufficiently large $n\in A.$ Inequalities \eqref{X2} imply
$$\frac{b_n}{a_n}\le\frac{a_{n+1}}{\tau_{n+1}}.$$ Hence
$$\infty=\lim_{n\to\infty, n\in A}\frac{b_n}{a_n}\le\limsup_{n\to\infty, n\in A}\frac{a_{n+1}}{\tau_{n+1}}\le\limsup_{n\to\infty}\frac{a_{n+1}}{\tau_{n+1}},$$
contrary to Lemma~\ref{Lem2.3}.
\end{proof}
\begin{proposition}\label{Pr2.10}
Let $E\subseteq\mathbb R^{+},$ $\tilde \tau\in\tilde
E_{0}^{d},$ and let $\{(a_n^{(1)},b_n^{(1)})\}_{n\in\mathbb N},$
$\{(a_n^{(2)},b_n^{(2)})\}_{n\in\mathbb N}$ be two sequences belonging to $\tilde
I_{E}.$ If $\tilde a^{1}\asymp\tilde\tau$ and $\tilde a^{2}\asymp\tilde\tau,$ where
$\tilde a^{i}:=\{a_n^{(i)}\}_{n\in\mathbb N}, \, i=1,2,$ then there is $N_0\in\mathbb N$
such that
\begin{equation}\label{L11}
(a_{n}^{(2)},b_{n}^{(2)} )=(a_{n}^{(1)}, b_{n}^{(1)})
\end{equation}
for every $n\ge N_0.$
\end{proposition}
\begin{proof}
Let us denote by $E^{1}$ the closure of $E$ in $\mathbb R^{+}.$ Using Remark~\ref{rem*} we see that $0\in ac E^{1}$ and $\tilde\tau\in\tilde E_{0}^{1d}.$  Since the sequences $\{(a_n^{(i)},b_n^{(i)})\}_{n\in\mathbb N}, \, i=1,2,$ belong to $\tilde I_{E},$ they
also belong to $\tilde I_{E^1}.$ By Lemma~\ref{Lem2.9}, we obtain $\tilde a^{i}\in\tilde
E_{0}^{1d}, \, i=1,2.$ We also have $\tilde\tau\asymp \tilde a^1,$ and $\tilde\tau\asymp
\tilde a^2.$ Consequently the weak equivalence $\tilde a^1 \asymp \tilde a^2$ holds.
Applying Lemma~\ref{Lem2.3} we can find $N_0\in\mathbb N$ such that $a_n^{(1)}\le
a_n^{(2)}$ and $a_n^{(2)}\le a_n^{(1)}$ for $n\ge N_{0}.$ Consequently
$a_n^{(1)}=a_n^{(2)}$ for $n\ge N_0$ which implies \eqref{L11} for such $n.$
\end{proof}
Define the set $\tilde I_{E}^{d}\subseteq\tilde I_{E}$ by the rule
$$(\{(a_n,b_n)\}_{n\in\mathbb N}\in\tilde I_{E}^{d})\Leftrightarrow(\{(a_n, b_n)\}_{n\in\mathbb N}\in\tilde I_{E} \mbox{\, and \,}  \{a_n\}_{n\in\mathbb N} \mbox{\, is almost decreasing}).$$

\begin{remark}\label{rem2.13}
Let $E\subseteq\mathbb R^{+}.$ If $\{(a_n,
b_n)\}_{n\in\mathbb N}\in\tilde I_{E}^{d},$ then there are
$\tilde\tau=\{\tau_n\}_{n\in\mathbb N}\in\tilde E_{0}^{d}$ and $\tilde \beta
=\{\beta_n\}_{n\in\mathbb N}\in\tilde E_0^{d}$ such that
\begin{equation}\label{y*}\mathop{\lim}\limits_{n\to\infty}\frac{\tau_n}{a_n}=\mathop{\lim}\limits_{n\to\infty}\frac{\beta_n}{b_n}~=~1.\end{equation}
\end{remark}

\begin{definition}\label{univ}Let $\tilde A:=\{(a_n, b_n)\}_{n\in\mathbb
N}~\in~\tilde I_{E}^{d}\quad\mbox{and}\quad \tilde L:=\{(l_n,m_n)\}_{n\in\mathbb
N}~\in~\tilde I_{E}^{d}.$ We write $\tilde A\preceq\tilde L$ if there are a natural
number $N_1 =N_1(\tilde A,\tilde L)$ and a function $f: \mathbb N_{N_{1}}\rightarrow\mathbb N,$ where $\mathbb N_{N_1}:=\{N_1, N_1 +1,...\},$ such that
\begin{equation}\label{L12}a_n = l_{f(n)}\end{equation} for every $n\in\mathbb N_{N_1}.$
We say that $\tilde L\in\tilde I_{E}^{d}$ is universal if $\tilde A \preceq \tilde
L$ for every $\tilde A\in\tilde I_{E}^{d}.$ \end{definition}

The first part of Definition~\ref{univ} can be reformulated as the following.
\begin{proposition}\label{univ*} Let $\tilde A=\{(a_n, b_n)\}_{n\in\mathbb
N}$ and $\tilde L=\{(l_n,m_n)\}_{n\in\mathbb N}$ belong to $\tilde I_{E}^{d}.$ $\tilde
A\preceq \tilde L$ if and only if there are $N_1 = N_{1}(\tilde A, \tilde L))$ and $f: \mathbb
N_{N_1}\rightarrow\mathbb N$ such that $$b_n = m_{f(n)} \,\, \mbox{for}\,\, n\in\mathbb
N_{N_1}.$$
\end{proposition}
\begin{remark}\label{refom}
The universality of $\tilde L\in\tilde I_{E}^{d}$ can be expressed in the language of arrows.
Let us denote by $Com$ the set of the connected components of $Ext E$. An element
$\tilde L\in\tilde I_{E}^{d}$ is universal if for every $\tilde A\in\tilde I_{E}^{d}$
there are $N_1 \in\mathbb N$ and $f:\mathbb N_{N_1}\rightarrow \mathbb N$ such the diagram
\begin{equation*}
\xymatrix{ \mathbb{N}_{N_1}\ar[r]^{in}
 \ar[dr]^f &\mathbb{N}\ar[r]^{\tilde{A}}& Com \\
&\mathbb{N}\ar[ur]^{\tilde{L}} }
\end{equation*}
is commutative. Here $in$ is the natural inclusion of $\mathbb
N_{N_{1}}$ in $\mathbb N,$ $in(n)=n$ for $n\in\mathbb N_{N_1}.$
\end{remark}
The following proposition describes the universal elements as the largest elements of the suitable posets.
\begin{proposition}\label{Pr2.11}
Let $E\subseteq\mathbb R^{+}$ be strongly porous on the right at 0 and let $0\in ac E.$ The relation $\preceq$ is a preorder on the set $\tilde
I_{E}^{d}.$
\end{proposition}
\begin{proof}
We must show that $\preceq$ is reflexive and transitive. The reflexivity of $\preceq$ is
evident. To prove that $\preceq$ is transitive note that if $\tilde A \preceq \tilde L,$
then there is an \emph{increasing} function $f: \mathbb N_{N_1}\rightarrow \mathbb N$
such that \eqref{L12} holds. (The existence  of an increasing $f$ meeting \eqref{L12}
follows because the sequences $\{a_n\}_{n\in\mathbb N}$ and $\{l_n\}_{n\in\mathbb N}$
are almost decreasing.) Suppose that $\tilde A\preceq\tilde L$ and $\tilde
L\preceq\tilde T,$ $\tilde T=\{(t_n, p_n)\}_{n\in\mathbb N}\in\tilde I_{E}^{d}.$ Let $f:
\mathbb N_{N_1}\rightarrow\mathbb N$ and $g:\mathbb N_{N_2}\rightarrow\mathbb N$ be two
functions such that $$a_n=l_{f(n)} \quad \mbox{for} \quad n\ge N_1 \quad \mbox{and}
\quad l_n=t_{g(n)} \quad \mbox{for}\quad n\ge N_2.$$ Put $M:=\max\{n\in\mathbb N:
f(n)\le N_2\}.$ Since $f$ is increasing and unbounded, we have $M<\infty.$ Define
$$N_3:=\max\{M, N_{1}\}$$ with $N_3:=N_1$ if $\{n\in\mathbb N: f(n)\le N_2\}=\varnothing.$ Then
the inequality $N_3<\infty$ holds. In accordance with the construction, we have $f(n)\ge
N_2$ for every $n\in\mathbb N_{N_3}.$ Consequently we obtain
$$a_n=l_{f(n)}=t_{g(f(n))}$$ for such $n.$ Thus $\tilde A\preceq\tilde L$ and $\tilde
L\preceq\tilde T$ imply $\tilde A\preceq\tilde T.$ \end{proof} Using the standard facts
from the ordered sets theory we may prove that the preorder $\preceq$ generates an
equivalence $\equiv$  on $\tilde I_{E}^{d}$ if we put
\begin{equation}\label{L12*}(\tilde A\equiv \tilde T)\Leftrightarrow(\tilde
A\preceq\tilde T \,\, \mbox{and} \,\, \tilde T\preceq\tilde A).\end{equation} Moreover, if the
relations $\tilde A\equiv\tilde S$ and $\tilde L \equiv \tilde T$ hold, then $\tilde
A\preceq\tilde L$ if and only if $\tilde S\preceq\tilde T.$ Going over to the factor set induced by $\equiv$ we obtain a partially
ordered set (poset). \emph{The preordered set $(\tilde I_{E}^{d}, \preceq)$ has an
universal element if and only if this poset has the largest element.}

Let $\tilde L=\{(l_n, m_n)\}_{n\in\mathbb N}\in\tilde I_{E}^{d}$ be universal. Let us
define the quantity
\begin{equation}\label{L13}
M=M(\tilde L):=\limsup_{n\to\infty}\frac{l_n}{m_{n+1}}.
\end{equation}

We say that a sequence $\tilde a=\{a_n\}_{n\in\mathbb N},$ $a_n \in \mathbb R,$ is
\emph{almost strictly decreasing} if $a_{n+1} < a_n$ for sufficiently large $n$. Write
$\tilde I_{E}^{sd}$ for the set of $\{(a_n, b_n)\}_{n\in\mathbb N}\in\tilde I_{E}^{d}$
having almost strictly decreasing $\{a_n\}_{n\in\mathbb N}.$
\begin{theorem}\label{ImpTh}
Let $E\subseteq\mathbb R^{+}$ be strongly porous on the right at 0 and let $0\in ac E.$  The following conditions are equivalent.
\item[\rm(i)]\textit{$E$ is a \textbf{CSP} - set.}
\item[\rm(ii)]\textit{The preodered set $(\tilde I_{E}^{d}, \preceq)$ contains an universal element $\tilde L=\{(l_n, m_n)\}_{n\in\mathbb N}\in \tilde I_{E}^{sd}$ with \begin{equation}\label{L14} M(\tilde L)<\infty.\end{equation}}
\item[\rm(iii)]\textit{$E$ is uniformly strongly porous.}
\end{theorem}
To prove Theorem~\ref{ImpTh} we need some additional lemmas.
\begin{lemma}\label{lem2.14}
Let $E\subseteq\mathbb R^{+}.$ If $\tilde L=\{(l_n, m_n)\}_{n\in\mathbb N}\in\tilde I_{E}^{d}$ is universal, then there is a subsequence
$\tilde L'=\{(l_{n_{k}}, m_{n_{k}})\}_{k\in\mathbb N}$ of $\tilde L$ such that $\tilde
L'$ is also universal and $\tilde L'\in\tilde
I_{E}^{sd}.$
\end{lemma}
\begin{proof}
We construct $\tilde L'$ by induction. Since $\{l_n\}_{n\in\mathbb N}$ is almost
decreasing, there exists $n_1 \in\mathbb N$ such that $l_{n+1}\le l_n$ for $n\ge n_1.$
The limit relation $\mathop{\lim}\limits_{n\to\infty}l_n =0$ implies that there is $n\ge
n_1$ such that $l_n < l_{n_1}.$ Write $$n_2: =\min\{n\in\mathbb N_{n_1}: l_n <
l_{n_1}\}.$$ Similarly we set
\begin{equation}\label{L15}
n_{k+1}:=\min\{n\in\mathbb N_{n_k}: l_n < l_{n_k}\}
\end{equation} for $k=2,3,4... \, .$
For every $n\ge n_1$ there is the unique $k\in\mathbb N$ such that
\begin{equation}\label{L16}
n_{k}\le n < n_{k+1}.
\end{equation}
Furthermore, the decrease of the sequence $\{l_n\}_{n\in\mathbb N_{n_1}}$ implies that
\begin{equation}\label{L17}
l_{n_{k}}=l_n
\end{equation}if $n$ satisfies \eqref{L16}. Let us define $g:\mathbb N_{n_1}\rightarrow \mathbb
N$ by the rule $g(n)=k$ where $k$ is the unique index satisfying \eqref{L16}. In fact, it was proved above that
$\tilde L \preceq \tilde L'.$ By Proposition~\ref{Pr2.11} the relation $\preceq$
is transitive. Since $\tilde L$ is universal, we have $\tilde T\preceq\tilde L$ for
every $\tilde T \in \tilde I_{E}^{d}.$ Consequently $\tilde T\preceq\tilde L'$ for every
$\tilde T\in\tilde I_{E}^{d},$ i.e., $\tilde L'$ is universal. It still remains to note
that \eqref{L15} implies the inequality $l_{n_k}>l_{n_{k+1}}$ for every $k\in\mathbb N.$
Hence $\{l_{n_k}\}_{k\in\mathbb N}$ is a strictly decreasing sequence. Thus $\tilde
L'\in\tilde I_{E}^{sd}.$
\end{proof}
\begin{remark}\label{rm2.19*}
If $\tilde L = \{(l_n, m_n)\}_{n\in\mathbb N}\in \tilde I_{E}^{sd}$ and $\tilde A =
\{(a_n, b_n)\}_{n\in\mathbb N}\in \tilde I_{E}^{sd},$  then
 \emph{$\tilde L\equiv\tilde A$ if and only if there exist $N_1,\, N_2\in \mathbb N$ such that $$(l_{n+N_{1}}, m_{n+N_{1}})=(a_{n+N_{2}}, b_{n+N_{2}})$$
for every $n\in\mathbb N,$ where $\equiv$ is defined by \eqref{L12*}.}

\bigskip We do not use this affirmation in the sequel and omit the proof here.
\end{remark}
\begin{lemma}\label{Lem2.15}
Let $E$ be a \textbf{CSP} - set. If $\tilde L=\{(l_n, m_n)\}_{n\in\mathbb N}\in\tilde
I_{E}^{sd}$ is universal,
then
\begin{equation}\label{L18}
M(\tilde L)=C(E)
\end{equation}
where the quantities $M(\tilde L)$ and $C(E)$ are defined by \eqref{L13} and \eqref{L9}
respectively.
\end{lemma}
\begin{proof}
Let $\tilde L\in\tilde I_{E}^{sd}$ be universal. We shall first prove the inequality
\begin{equation}\label{L19}
M(\tilde L)\ge C(E).
\end{equation}
Inequality \eqref{L19} holds if and only if
\begin{equation}\label{L20}
M(\tilde L)\ge C(\tilde \tau)
\end{equation} for every $\tilde \tau\in \tilde E_{0}^{d},$ where $C(\tilde \tau)$ was defined in \eqref{L9}. Let $\tilde\tau \in \tilde E_{0}^{d}. $ By condition of the lemma, $E$ is completely strongly porous. Hence there is $\{(a_n,
b_n)\}_{n\in\mathbb N}\in\tilde I_{E}$ such that $\tilde\tau\asymp\tilde a.$ By
Lemma~\ref{Lem2.3} we have the inequality
\begin{equation}\label{L21}
\limsup_{n\to\infty}\frac{a_n}{\tau_n}<\infty
\end{equation} and, for sufficiently large $n$, the inequality
\begin{equation}\label{L22}
\tau_n \le a_n.
\end{equation} Proposition~\ref{Pr2.10} and the definition of
$C(\tilde\tau)$ imply
\begin{equation}\label{L23}
C(\tilde\tau)=\limsup_{n\to\infty}\frac{a_n}{\tau_n}.
\end{equation}
Hence to prove \eqref{L20} we must show that
\begin{equation}\label{L24}
M(\tilde L)\ge\limsup_{n\to\infty}\frac{a_n}{\tau_n}.
\end{equation}
Be Lemma~\ref{Lem2.9}  we have
\begin{equation}\label{L25}
\tilde A:=\{(a_n, b_n)\}_{n\in\mathbb N}\in\tilde I_{E}^{d}.
\end{equation}
Since $\tilde L$ is universal, from
\eqref{L25} follows that $\tilde A \preceq \tilde L.$ Consequently there are $N_1 \in
\mathbb N$ and the increasing function $f: \mathbb N_{N_1}\rightarrow\mathbb N$ such that
\begin{equation}\label{L26}
a_n \ge a_{n+1} \quad \mbox{and} \quad a_n =l_{f(n)}
\end{equation} for $n\ge N_1.$ Since $\tilde L=\{(l_n, m_n)\}_{n\in\mathbb N}\in\tilde I_{E}^{sd},$ we may suppose that $\tilde l=\{l_n\}_{\in\mathbb N}$ is strictly decreasing. Replacing $\tilde\tau$ by a suitable subsequence we
may assume that $\tilde\tau$ and $ \tilde a$ are also strictly decreasing, $f$ is
strictly increasing, and that the relations
\begin{equation}\label{L27}
\tau_1 \le l_1, \quad
\lim_{n\to\infty}\frac{a_n}{\tau_n}=\limsup_{n\to\infty}\frac{a_n}{\tau_n}
\end{equation} hold. The closed intervals $[m_{n+1}, l_n], \, n=1,2,...,$ together with
the half-open interval $[m_{1}, \infty)$ form a cover of the set
$E_{0}=E\setminus\{0\},$ i.e. $$E_{0}\subseteq[m_1,
\infty)\cup\left(\bigcup_{n\in\mathbb N}[m_{n+1}, l_n]\right).$$ Since the elements of
this cover are pairwise disjoint and $\tau_1 \le l_1,$ for every $n\in\mathbb N$ there
is a unique $k(n)\in\mathbb N$ such that
\begin{equation}\label{L28}
\tau_n\in[m_{k(n)+1}, l_{k(n)}].
\end{equation} We claim that the equality
\begin{equation}\label{L29}
k(n)=f(n)
\end{equation} holds for sufficiently large $n.$ Indeed, using \eqref{L22}, \eqref{L26} and
\eqref{L28} we obtain $$\tau_n\le l_{f(n)} \quad \mbox{and} \quad \tau_n \ge
m_{k(n)+1}.$$ These inequalities and $$m_{k(n)+1}>l_{k(n)+1}>l_{k(n)+2}>l_{k(n)+3}>...
$$ imply
\begin{equation}\label{L30}
f(n)\le k(n).
\end{equation}
Suppose that the last inequality is strict for $n$ belonging to an infinite set
$A\subseteq\mathbb N,$ i.e.
\begin{equation}\label{L31}
f(n)\le k(n)-1
\end{equation} for $n\in A.$ Since $\{a_n\}_{n\in\mathbb N}\asymp\{\tau_n\}_{n\in\mathbb
N}$ and $a_n = l_{f(n)},$ we can find a constant $c \in (0,1)$ such that
\begin{equation}\label{L32}
cl_{f(n)}\le \tau_n \le l_{f(n)}
\end{equation} for sufficiently large  $n.$ From \eqref{L28}, \eqref{L30} and \eqref{L32} it follows that
\begin{equation}\label{L33}
cl_{f(n)}\le \tau_n \le l_{k(n)}\le l_{f(n)}.
\end{equation}
Since $\tilde l=\{l_n\}_{n\in\mathbb N}$ is strictly increasing and $(l_n, m_n)\cap
(l_j, m_j)=\varnothing$ if $n\ne j,$ \eqref{L31} implies that $$l_{k(n)}<m_{k(n)}\le
l_{k(n)-1}\le l_{f(n)}<m_{f(n)}.$$ These inequalities and \eqref{L33} show that
$$cl_{f(n)}\le \tau_n \le l_{k(n)}<m_{k(n)}\le l_{k(n)-1}<l_{f(n)}$$ for $n\in A.$
Consequently we have
$$\frac{1}{c}=\lim_{n\to\infty}\frac{l_{f(n)}}{cl_{f(n)}}\ge\limsup_{n\to\infty, \, n\in
A}\frac{m_{k(n)}}{l_{k(n)}},$$ contrary to the limit relation
$$\lim_{n\to\infty}\frac{m_n}{l_n}=\infty.$$ Hence the set
of $n\in\mathbb N$ meeting the condition $f(n)<k(n)$ is finite. Thus \eqref{L29} holds for sufficiently large
$n.$

Now it is easy to prove \eqref{L24}. By \eqref{L26} and \eqref{L29} we have $$a_n =
l_{f(n)}=l_{k(n)}.$$ Relation \eqref{L28} implies $\tau_n \ge m_{k(n)+1}.$
Consequently
$$\frac{a_n}{\tau_n}\le \frac{l_{k(n)}}{m_{k(n)+1}}.$$ Hence $$\limsup_{n\to\infty}\frac{a_n}{\tau_n}\le \limsup_{n\to\infty}\frac{l_{k(n)}}{m_{k(n)+1}}\le \limsup_{n\to\infty}\frac{l_n}{m_{n+1}}=M(\tilde L).$$
Inequality \eqref{L24} follows, so that \eqref{L19} is proved.

To prove the inequality
\begin{equation}\label{L34}
 M(\tilde L) \le C(E)
\end{equation} we take a sequence $\tilde\tau=\{\tau_n\}_{n\in\mathbb N}\in\tilde E_{0}^{d}$
such that \eqref{L28} holds with $k(n)=n$ and
\begin{equation}\label{L35}
\lim_{n\to\infty}\frac{m_{n+1}}{\tau_n}=1.
\end{equation}
A desirable $\tilde\tau$ can be constructed as in the proof of Proposition~\ref{P1}. The
set $E$ is $\tilde\tau$-strongly porous because $E$ is a \textbf{\emph{CSP}} - set.
Hence there is $\tilde a\asymp \tilde \tau$ such that $\{(a_n , b_n)\}_{n\in\mathbb
N}\in\tilde I_{E}^{d}.$ By Lemma~\ref{Lem2.9} the sequence $\tilde a$ is almost
decreasing. Since $\tau_n \in [m_{n+1}, l_n],$
using \eqref{L29} we obtain
$$a_n = l_n$$ for sufficiently large $n.$ From \eqref{L23} and \eqref{L35} it follows
that
\begin{equation*}
C(\tilde\tau)=\limsup_{n\to\infty}\frac{a_n}{\tau_n}=\limsup_{n\to\infty}\frac{l_n}{m_{n+1}}\frac{m_{n+1}}{\tau_n}=\limsup_{n\to\infty}\frac{l_n}{m_{n+1}}\lim_{n\to\infty}\frac{m_{n+1}}{\tau_n}
\end{equation*}
\begin{equation}\label{L36}
=\limsup_{n\to\infty}\frac{l_n}{m_{n+1}}=M(\tilde L).
\end{equation}
Since $C(E)\ge C(\tilde\tau),$ inequality \eqref{L34} follows.

To complete the proof, it suffices to observe that \eqref{L19} and \eqref{L34} imply
\eqref{L18}. \end{proof} Directly from \eqref{L36} we obtain
\begin{corollary}\label{Col2.16}
Let $E\subseteq\mathbb R^{+}$ be a \textbf{CSP} - set. If $\tilde L=\{(l_n, m_n)\}_{n\in\mathbb N}\in\tilde
I_{E}^{sd}$ is universal, then $M(\tilde
L)<\infty.$
\end{corollary}
\begin{remark}\label{Rem2.17}
As has been shown in Lemma~\ref{Lem2.15} the equality $M(\tilde L)=C(E)$ holds for every
universal $\tilde L\in\tilde I_{E}^{sd}.$ Suppose that $\tilde L\in\tilde I_{E}^{d}$ is
universal but $\tilde L\not\in \tilde I_{E}^{sd}.$ Define the set $A\subseteq\mathbb N$
by the rule $$(n\in A)\Leftrightarrow (n\in\mathbb N \,\, \mbox{and} \,\,
(l_{n+1},m_{n+1})=(l_{n},m_{n})).$$ Let $\tilde L^{'}\in\tilde I_{E}^{sd}$ be the
universal element of $(\tilde I_{E}^{d},\preceq)$ constructed from $\tilde L$ as in
Lemma~\ref{lem2.14}. Using the definition of the set $A$ we obtain $$M(\tilde
L)=\limsup_{n\to\infty}\frac{l_{n+1}}{m_{n}}=\limsup_{n\to\infty, \, n\in
A}\frac{l_{n+1}}{m_{n}}\vee\limsup_{n\to\infty, \, n\in \mathbb N\setminus
A}\frac{l_{n+1}}{m_{n}}$$ $$=\limsup_{n\to\infty,\, n\in A}\frac{l_{n}}{m_{n}}\vee
M(\tilde L')=0\vee M(\tilde L')=M(\tilde L').$$ Consequently if $\tilde L, \tilde S \in
\tilde I_{E}^{d}$ are universal, then $M(\tilde L)=M(\tilde S).$ Thus condition (ii) of Theorem~\ref{ImpTh}
can be formulated by the following equivalent way.

\bigskip
\noindent $\bullet$ \emph{The set of universal elements $\tilde L\in\tilde I_{E}^{d}$ is
nonempty and the inequality $M(\tilde L)<\infty$ holds for every $\tilde L$ from this
set.}\end{remark} \emph{Proof of Theorem~\ref{ImpTh}}. $\mathrm{(i)\Rightarrow (ii)}.$
Let $E$ be a \textbf{\emph{CSP}} - set. We shall first prove that there is a
sequence $\tilde u=\{u_n\}_{n\in\mathbb N}\in\tilde E_{0}^{d}$ such that for every
$\tilde \tau~=~\{\tau_k\}_{k\in\mathbb N}\in\tilde E_{0}^{d}$ can be found an almost
increasing function $f:\mathbb N\rightarrow\mathbb N$ satisfying the relation
\begin{equation}\label{L37}
\{\tau_k\}_{k\in\mathbb N}\asymp\{u_{f(k)}\}_{k\in\mathbb N}.
\end{equation} Let us define the sequence of sets $E_j, \, j\in\mathbb N,$ by the rule
\begin{equation}\label{L38}
E_{1}:=E\cap[1, \infty), \, E_{2}:=E\cap[\frac{1}{2}, 1),...,
E_{j}:=E\cap[\frac{1}{2^{j-1}}, \frac{1}{2^{j-2}}).
\end{equation} There is the unique subsequence $\{E_{j_n}\}_{n\in\mathbb N}$ of the
sequence $\{E_{j}\}_{j\in\mathbb N}$ such that $$
E\setminus\{0\}=\bigcup_{n\in\mathbb N}E_{j_n} \quad \mbox{and} \quad
E_{j_n}\ne\varnothing$$ for every $n\in\mathbb N.$ For convenience we set $A_n: =
E_{j_n}, \, n\in\mathbb N.$ Let $\{u_n\}_{n\in\mathbb N}$ be a sequence of positive real
numbers meeting the condition $u_n \in A_n$ for every $n\in\mathbb N.$ It is clear that
$\{u_n\}_{n\in\mathbb N}\in\tilde E_{0}^{d}.$ For every
$\tilde\tau=\{\tau_k\}_{k\in\mathbb N}\in\tilde E_{0}^{d},$ define $f: \mathbb
N\rightarrow\mathbb N$ by the rule $$f(k)=n \quad \mbox{if and only if} \quad \tau_k \in A_n.$$ The
function $f$ is well-defined because
$$E\setminus\{0\} =\bigcup_{n\in\mathbb N}A_n \quad \mbox{and} \quad A_j \cap A_i=\varnothing \quad
\mbox{if} \quad i\ne j.$$ It follows directly from \eqref{L38} that
\begin{equation*}\label{L39}
\frac{1}{2}\tau_k \le u_{f(k)}\le2 \tau_k
\end{equation*} if $f(k)\ge 2.$ Moreover, since $\tilde\tau$ and $\tilde u$ are almost
decreasing and $\mathop{\lim}\limits_{n\in\mathbb N}\tau_n =0,$ the function
$f: \mathbb N\rightarrow\mathbb N$ is almost increasing and the set $\{k\in\mathbb N:
f(k)=1\}$ is finite. Consequently there are some constants $c_1, \, c_2$ such that
\begin{equation*}\label{L40}
c_2\tau_k\le u_{f(k)}\le c_1 \tau_k
\end{equation*} for all $k\in\mathbb N$. Thus \eqref{L37} holds.

Let $\{u_n\}_{n\in\mathbb N}\in\tilde E_{0}^{d}$ be the sequence constructed above. Since $E$ is a \textbf{\emph{CSP}} - set,  $E$ is $\tilde
u$-strongly porous. Hence, there is $\tilde A:=\{(a_n, b_n)\}_{n\in\mathbb N}\in\tilde
I_{E}$ such that
\begin{equation}\label{L41}
\tilde a\asymp\tilde u.
\end{equation} Lemma~\ref{Lem2.9} implies that $\tilde a$ is almost decreasing, i.e., $\tilde A\in\tilde I_{E}^{d}.$ We claim that $\tilde A$ is universal.
Indeed, as was shown  for every $\tilde \tau=\{\tau\}_{k\in\mathbb N}\in\tilde E_{0}^{d}$ there is $f: \mathbb
N\rightarrow\mathbb N$ such that \eqref{L37} holds.
The relation $\{u_n\}_{n\in\mathbb N}\asymp\{a_{n}\}_{n\in\mathbb N}$ implies that
\begin{equation}\label{L43}
\{u_{f(k)}\}_{k\in\mathbb N}\asymp\{a_{f(k)}\}_{k\in\mathbb N}.
\end{equation} Every interval $(a_{f(n)}, b_{f(n)})$ is a connected component of $Ext E$ and, in addition,
$\mathop{\lim}\limits_{n\to\infty}\frac{b_n}{a_n}=\infty$ implies
$\mathop{\lim}\limits_{k\to\infty}\frac{b_{f(k)}}{a_{f(k)}}=\infty$ because
$\mathop{\lim}\limits_{n\to\infty}f(n)=\infty.$ Consequently we obtain
\begin{equation}\label{L44}
\{(a_{f(k)}, b_{f(k)})\}_{k\in\mathbb N}\in\tilde I_{E}.
\end{equation} Moreover, since $f$ is almost increasing and $\tilde A=\{(a_n, b_n)\}_{n\in\mathbb N}\in\tilde
I_{E}^{d},$ \eqref{L44} implies
\begin{equation}\label{L45}
\{(a_{f(k)}, b_{f(k)})\}_{k\in\mathbb N}\in\tilde I_{E}^{d}.
\end{equation} From \eqref{L37} and \eqref{L43} we obtain
\begin{equation}\label{L46}
\{\tau_k\}_{k\in\mathbb N}\asymp\{a_{f(k)}\}_{k\in\mathbb N}.
\end{equation} Using \eqref{L45}, \eqref{L46} and Remark~\ref{rem2.13}, we can prove
that $\tilde L\preceq\tilde A$ for every $\tilde L\in\tilde I_{E}^{d},$ as required.

By Lemma~\ref{lem2.14} we can find an universal element $\tilde L\in\tilde I_{E}^{sd}.$
In according with Corollary~\ref{Col2.16} we have $M(\tilde L)<\infty.$ Thus condition
$\mathrm{(i)}$ implies $\mathrm{(ii)}.$

The implication $\mathrm{(iii)\Rightarrow (i)}$ is evident. Moreover, using
Lemma~\ref{Lem2.15}, we can simply verify that the implication $\mathrm{((i)\&
(ii))}\Rightarrow \mathrm{(iii)}$ is true. Consequently to complete the proof is
suffices to show that $\mathrm{(ii)}\Rightarrow \mathrm{(i)}.$ Suppose that  condition
$\mathrm{(ii)}$ holds. Let $\tilde\tau=\{\tau_n\}_{n\in\mathbb N}\in \tilde E_{0}^{d}$
and let $\tilde L=\{(l_k, m_k)\}_{k\in\mathbb N}\in\tilde I_{E}^{sd}$ be universal.
As in the proof of Lemma~\ref{Lem2.15} we may suppose that $\{l_n\}_{n\in\mathbb N}$ is a strictly
decreasing sequence and that $\tau_1 \le l_1 .$ Then for every $n\in\mathbb N$ there is
a unique $k(n)\in\mathbb N$ such that
\begin{equation}\label{L47}
m_{k(n)+1}\le\tau_{n}\le l_{k(n)},
\end{equation} (see \eqref{L28}). Double inequality \eqref{L47} implies $$\limsup_{n\to\infty}\frac{l_{k(n)}}{\tau_n}\le\limsup_{n\to\infty}\frac{l_{k(n)}}{m_{k(n)+1}}\le\limsup_{k\to\infty}\frac{l_k}{m_{k+1}}=M(\tilde L)<\infty.$$
Since $\{(l_{k(n)}, m_{k(n)})\}_{n\in\mathbb N}\in\tilde I_{E}^{d},$ the set $E$ is
$\tilde\tau$-strongly porous by Lemma~\ref{Lem2.3}. Thus condition $\mathrm{(i)}$
follows from condition $\mathrm{(ii)}.$ $\qquad \qquad \qquad \qquad \qquad \qquad
\qquad \qquad \qquad \qquad \, \Box$

\begin{remark}\label{2.23*}
Conditions (i) and (iii) from Theorem~\ref{ImpTh} are equivalent for arbitrary $E\subseteq \mathbb R^{+}.$ Indeed, if $p^{+}(E,0)<1,$ then both (i) and (iii) are evidently false. If $p^{+}(E,0)=1$ but $0\not\in ac E,$ then (i) and (iii) are true (see Remark~\ref{1.3*} and Remark~\ref{2}). In this connection it should be pointed out that condition (ii) of Theorem~\ref{ImpTh} implies $\tilde I_{E}\ne\varnothing.$ Consequently, if (ii) holds, then $0\in ac E$ and $p^{+}(E,0)=1$ (see Remark~\ref{1.3**}).
\end{remark}

The following example shows that the existence of an universal $\tilde L \in \tilde
I_{E}^{sd}$ does not imply the inequality $M(\tilde L)<\infty.$
\begin{example}\label{ex3}
Let $\{x_n\}_{n\in\mathbb N}$ be a sequence of strictly decreasing positive real numbers
such that $$\lim_{n\to\infty}\frac{x_{n+1}}{x_{n}}=0.$$ Define the set $E$ as
$\mathop\bigcup\limits_{n\in\mathbb N}[x_{2n+1}, x_{2n}].$ It follows from Lemma~\ref{Lem2.3} that $E\not\in\textbf{\emph{CSP}}$ but, as easily seen, the sequence
$\{(x_{2n}, x_{2n-1})\}_{n\in\mathbb N}~\in~\tilde I_{E}^{sd}$ is universal. (See Fig.~1.).
\end{example}

\begin{figure}[h]
\centerline{\includegraphics[scale=0.7]{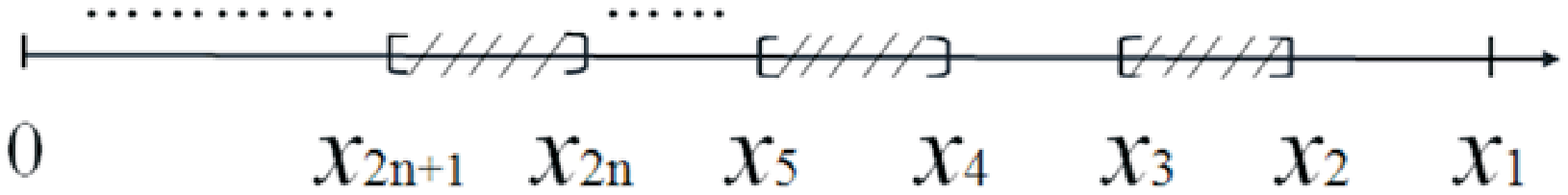}}%
\centerline{\emph{Fig. 1. The set E is shaded here}}
\end{figure}

The next theorem
describes the structure of sets $E\subseteq\mathbb R^{+}$ for which there is an universal $\tilde L\in\tilde
I_{E}^{sd}.$

As in Remark~\ref{refom} write $Com$ for the set of all connected components of $Ext E.$
\begin{theorem}\label{Th2.24} Let $E\subseteq\mathbb R^{+}$ be
strongly porous on the right at 0 and let $0 \in ac E.$ The preordered set
$(\tilde I_{E}^{d}, \preceq)$ contains an universal element if and only if there is a
constant $c>1$ such that for every $K>1$ there is $t>0$ for which the inequalities $t>a$
and $\frac{b}{a}>c$ imply the inequality $\frac{b}{a}>K$ for every $(a, b)\in Com.$
\end{theorem}
\begin{proof}
Suppose that there is an universal
element $\tilde L=\{(l_n, m_n)\}_{n\in\mathbb N}\in\tilde I_{E}^{d}.$ We must prove that
\begin{equation}\label{L48}
\exists \,\, c>1\, \forall\, K>1 \, \exists \,\, t>0 \, \forall \,(a,b) \in Com: (a<t)\&
\left(\frac{b}{a}>c\right)\Rightarrow\left(\frac{b}{a}>K\right).
\end{equation}
By Lemma~\ref{lem2.14} we may assume that $\{l_n\}_{n\in\mathbb N}$ is strictly
decreasing. Using the limit relations $$\lim_{n\to\infty}\frac{m_n}{l_n}=\infty \,\,\,
\mbox{and} \, \lim_{n\to\infty}l_{n}=0$$ and the strict decrease of
$\{l_n\}_{n\in\mathbb N}$ we obtain that
\begin{equation}\label{L49}
 \forall\, K>1 \, \exists \,\, t>0:\, \, \forall\, n\in\mathbb N\,\,  (l_n <t)\Rightarrow\left(\frac{m_n}{l_n}>K\right).
\end{equation}
If \eqref{L48} does not hold, then
\begin{equation}\label{L50}
\forall \,\, c>1\, \exists\, K=K(c)>1 \, \forall \,\, t>0 \, \exists \,(a,b) \in Com:
(t>a)\,\, \& \left(c<\frac{b}{a}\le K(c)\right).
\end{equation}
Using this formula with $c=j$ and $K=K(j),$ for $j=1,2,...,$ we see that
\begin{equation}\label{L51}
\forall \,\, t>0\,  \exists \,(a_j,b_j) \in Com: (a_j<t)\,\,  \& \left(j\le\frac{b_j}{a_j}\le
 K(j)\right).
\end{equation}
Formula \eqref{L49} implies that
\begin{equation}\label{L52}
\forall \,\, n\in\mathbb N \,\,\exists\,\, t_j >0: \, \, (l_n <t_j)\Rightarrow\left(\frac{m_n}{l_n}>K(j)\right).
\end{equation}
We can suppose also that $\mathop{\lim}\limits_{j\to\infty}t_{j}=0$ and $\{t_j\}_{j\in\mathbb N}$ is
strictly decreasing. From \eqref{L51} with $t=t_j$ it follows that
\begin{equation}\label{L53}
\forall \,\, j\in\mathbb N\,  \exists \,(a_j,b_j) \in Com: (a_j<t_j)\,\, \&
\left(j\le\frac{b_j}{a_j}\le
 K(j)\right).
\end{equation}
Consequently the sequence $\tilde A:=\{(a_j, b_j)\}_{j\in\mathbb N}$ belongs to $\tilde
I_{E}.$ Using the limit relation $\lim_{j\to\infty}t_{j}=0$ and passing on to a suitable
subsequence we can also claim that $\tilde A\in\tilde I_{E}^{d}.$ Formulas \eqref{L52}
and \eqref{L53} imply that $$(a_j, b_j)\ne (l_n, m_n)$$ for every element $(l_n, m_n)$
of $\tilde L.$ Consequently $\tilde L$ is not universal, contrary to the assumption.

Conversly, suppose that \eqref{L48} holds. Let us prove that there exists an universal
element in $\tilde I_{E}^{d}.$ Let $c$ be the constant satisfying \eqref{L48}. Define a subset $Com(c)$ of the set $Com$  by
the rule $$((a,b)\in Com(c))\Leftrightarrow \left((a,b)\in Com, \, a>1 \,\,
\mbox{and}\,\, \frac{b}{a}>c\right).$$ We
can enumerate of the intervals $(a,b)\in Com(c)$ in the sequence $$(a_1, b_1), \, (a_2,
b_2), \,..., (a_n, b_n), \,...$$ such that $\{a_n\}_{n\in\mathbb N}$ is strictly
decreasing. Condition \eqref{L48} implies that $\tilde A=\{(a_n,
b_n)\}_{n\in\mathbb N}\in\tilde I_{E}^{d}.$ The universality of $\tilde A$ follows
directly from Definition~\ref{univ} and \ref{L48}.
\end{proof}
As in Remark~\ref{2.23*} it should be noted that the existence of an universal $\tilde L\in\tilde I_{E}^{d}$ implies that $0\in ac E$ and $p^{+}(E,0)=1.$

\bigskip

\noindent \textbf{An illustrating model to Theorem~\ref{Th2.24}.} Let $E\subseteq(0,1]$
be closed and let $0\in ac E.$  Write $$W=-\ln
E:=\left\{\ln\left(\frac{1}{x}\right): x\in E \right\}.$$ We can consider $W$ as ``a photography of an one-dimensioned liquid'' with some ``gas bubbles'' $\left( \ln\left(\frac{1}{b}\right),
\ln\left(\frac{1}{a}\right) \right),$ where $(a,b)\in Com,$ which
move to $+\infty.$ Theorem~\ref{Th2.24} means that there is a critical value $\ln c$
such that if the size of gas bubbles are greater than $\ln c,$ then these bubbles
undergo an unbounded blow up during theirs motion.

\bigskip
The following simple proposition can be considered as a limit case of Theorem~\ref{Th2.24}.
\begin{proposition}\label{Pr2.27}
Let $E \subseteq\mathbb R^{+}$ and $\tilde L=\{(l_n, m_n)\}_{n\in\mathbb N}\in\tilde I_{E}^{d}.$ Suppose that for every $(a,b)\in Com$ there is $n\in\mathbb N$ such that $(a,b)=(l_n, m_n).$ Then $\tilde L$ is universal. 
\end{proposition}
The proof follows directly from Definition~\ref{univ}.


\section{Another characterizations of \textbf{\emph{CSP}} - sets}
\hspace*{\parindent} Let $E$ be a subset of $\mathbb R^{+}.$ Define the set $\tilde
H=\tilde H(E)$ of the sequences $\tilde h=\{h_n\}_{n\in\mathbb N},$ $h_n >0,$
$\mathop{\lim}\limits_{n\to\infty}h_n = 0$ by the rule:
\begin{equation}\label{l3.1}
(\tilde h\in\tilde H)\Leftrightarrow \left ( \frac{\lambda(E,0,h_n)}{h_n}\rightarrow
p^{+}(E,0) \quad \mbox{with} \quad n\rightarrow\infty\right)
\end{equation}
where the quantities $p^{+}(E,0)$ and $\lambda(E,0, h_n)$ are the same as in
Definition~\ref{D1}.
\begin{theorem}\label{th3.1}
Let $E\subseteq\mathbb R^{+}$ be strongly porous on the right at 0. Then $E$ is a \textbf{\emph{CSP}} - set if and only if for every $\tilde \tau=\{\tau_n\}_{n\in\mathbb N}\in
\tilde E_0^{d}$ there is $\tilde h=\{h_n\}_{n\in\mathbb N}\in\tilde H(E)$ such that
$\tilde\tau\asymp\tilde h.$
\end{theorem}
\noindent \emph{Proof.} The necessity is easy to prove. Suppose $E$ is a \textbf{\emph{CSP}} - set. Let $\tilde\tau \in \tilde E_0^{d}.$ By Theorem~\ref{ImpTh} there is an universal element $\tilde
L=\{(l_n, m_n)\}_{n\in\mathbb N}\in\tilde I_{E}^{sd}$ with $M(\tilde L)<\infty.$
Reasoning as in the proof of Theorem~\ref{ImpTh}, we can
find $k(n)\in\mathbb N$ such that
\begin{equation}\label{l3.2}
\tau_n \in [m_{k(n)+1}, l_{k(n)}]
\end{equation} for sufficiently large $n$ (see \eqref{L28}).
Membership \eqref{l3.2} implies the inequalities $$m_{k(n)+1}\le\tau_n \quad \mbox{and}
\quad \frac{\tau_n}{m_{k(n)+1}}\le\frac{l_{k(n)}}{m_{k(n)+1}}.$$ Thus we have
$$\limsup_{n\to\infty}\frac{\tau_n}{m_{k(n)+1}}\le \limsup_{n\to\infty}\frac{l_{k(n)}}{m_{k(n)+1}}\le M(\tilde L)<\infty.$$
Consequently there are $c_1 \ge 1$ and $N_1 \in\mathbb N$ such that
$m_{k(n)+1}\le\tau_n \le c_{1}m_{k(n)+1}$ for $n\ge N_1.$ If we set $m_{k(n)+1}:=m_{k(N_1)+1} \quad \mbox{for} \quad n<N_1,$
then it is easy to see that $\{\tau_n\}_{n\in\mathbb
N}\asymp\{m_{k(n)+1}\}_{n\in\mathbb N}.$ To be certain that
$\{m_{k(n)+1}\}_{n\in\mathbb N}\in\tilde H(E),$ it suffices to check that
\begin{equation}\label{l3.3}
\lim_{n\to\infty}\frac{\lambda(E,0,m_{k(n)+1})}{m_{k(n)+1}}=1.
\end{equation}
(Indeed, $p^{+}(E,0)=1$ because $E$ is strongly porous on the right at 0.) Since the quantity 
$\lambda(E,0,m_{k(n)+1})$ is the length of the largest open interval in the set $(0,
m_{k(n)+1})\cap Ext E$ and $$(l_{k(n)+1}, m_{k(n)+1})\subseteq(0, m_{k(n)+1})\cap Ext
E,$$ we have
\begin{equation}\label{l3.4}
\frac{m_{k(n)+1}-l_{k(n)+1}}{m_{k(n)+1}}\le\frac{\lambda(E,0,m_{k(n)+1})}{m_{k(n)+1}}\le
1.
\end{equation} The sequence $\tilde L$ belongs to $\tilde I_{E}^{sd}.$ Hence $$\lim_{n\to\infty}\frac{m_{k(n)+1}-l_{k(n)+1}}{m_{k(n)+1}}=1.$$
The last relation and \eqref{l3.4} imply \eqref{l3.3}.

The proof of the sufficiency is more awkward, so we divide it into several lemmas.
\begin{lemma}\label{lm3.2}
Let $E\subseteq\mathbb R^{+}$ be strongly porous on the right at 0 and let
$\tilde\tau=\{\tau_n\}_{n\in\mathbb N}\in\tilde E_{0}^{d}$ and $\tilde
h=\{h_n\}_{n\in\mathbb N}\in\tilde H(E).$ If $\tilde\tau\asymp\tilde h,$ then there is
$\{(a_n, b_n)\}_{n\in\mathbb N}\in\tilde I_{E}$ such that
\begin{equation}\label{l3.5}
\{\tau_n\}_{n\in\mathbb N}\asymp\{b_n\}_{n\in\mathbb N}.
\end{equation}
\end{lemma}
\begin{proof}
Let $\tilde\tau\asymp\tilde h.$ By the definition of $\tilde H(E),$ for every
$n\in\mathbb N,$ there is an interval $(a_{n}^{'},b_{n}^{'})\subseteq (0, h_n)\cap Ext
E$ such that
\begin{equation}\label{l3.6}
\lim_{n\to\infty}\frac{b_{n}^{'}-a_{n}^{'}}{h_n}=1.
\end{equation}
Moreover, the relation $\tilde\tau\asymp\tilde h$ implies that there are constants
$k\in(0,1)$ and $K\in(1, \infty)$ such that
\begin{equation}\label{l3.7}
\tau_n \in (kh_n, Kh_n)
\end{equation} for every $n\in\mathbb N.$ Consequently
\begin{equation}\label{l3.8}
\tau_n \in (0, Kh_n)\setminus (b_{n}^{'}-a_{n}^{'}).
\end{equation}
Using \eqref{l3.6} we can show that
\begin{equation}\label{l3.9}
b_{n}'>kh_n > a_{n}^{'}
\end{equation} for sufficiently large $n.$ It is clear that $Kh_n > h_n \ge b_{n}^{'}.$ Hence
\eqref{l3.7} -- \eqref{l3.9} imply
\begin{equation}\label{l3.10}
 \tau_n \in [b_{n}^{'}, Kh_n)
\end{equation} for sufficiently large $n$ (see Fig. 2 below).

\begin{figure}[h]
\centerline{\includegraphics[scale=0.7]{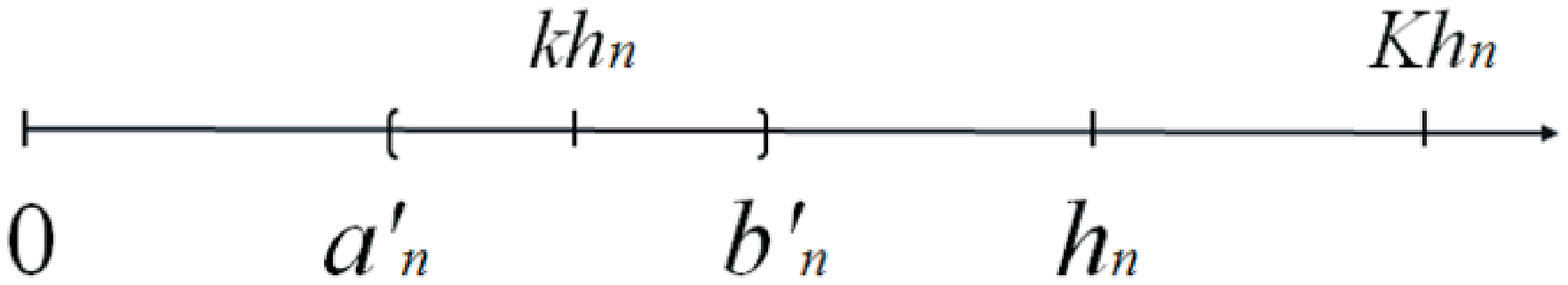}}%
\centerline{\emph{Fig. 2}}
\end{figure}
\noindent Let $(a_n, b_n)$ be the connected component of $Ext E$ meeting the inclusion
$(a_{n}^{'}, b_{n}^{'})\subseteq (a_n, b_n).$ From \eqref{l3.10} it follows $\tau_n \ge
b_n.$ Hence
\begin{equation}\label{l3.11}
kh_n < b_n^{'} \le b_n \le \tau_n <Kh_n
\end{equation}
for sufficiently large $n.$  Consequently $\tilde\tau\asymp\tilde h$ and $\tilde
b\asymp\tilde h,$ so that \eqref{l3.5} follows. To complete the proof, it suffices to
show the membership
$\{(a_n, b_n)\}_{n\in\mathbb N}\in\tilde I_{E}.$ The last relation holds if
and only if
\begin{equation}\label{l3.12}
\lim_{n\to\infty}\frac{a_n}{b_n}=0.
\end{equation}
Inequalities $a_n \le a_{n}' <b_{n}'\le b_{n}$ imply that
\begin{equation}\label{l3.13}
0\le\frac{a_n}{b_n}\le\frac{a_n^{'}}{b_n^{'}}.
\end{equation}
Moreover, since $$\frac{b_n^{'} - a_n^{'}}{h_n}\le \frac{b_n^{'} - a_n^{'}}{b_n^{'}}\le
1,$$ limit relation \eqref{l3.6} yields $$\lim_{n\to\infty}\frac{a_n^{'}}{b_n^{'}}=0.$$
Thus \eqref{l3.13} follows from \eqref{l3.12}.
\end{proof}
\begin{remark}\label{rm3.3}
It is clear that $\{b_n\}_{n\in\mathbb N}\in\tilde H(E)$ for $\{(a_n,
b_n)\}_{n\in\mathbb N}\in\tilde I_{E}.$
\end{remark}
The following lemmas are analogs of Lemma~\ref{Lem2.9}, Proposition~\ref{Pr2.10} and
have the similar proofs.
\begin{lemma}\label{lm3.5}
Let $E\subseteq\mathbb R^{+}.$ If
$\tilde\tau\in\tilde E_{0}^{d}$ and $\{(a_n, b_n)\}_{n\in\mathbb N}\in\tilde I_{E},$
then the weak equivalence $\tilde b\asymp\tilde\tau$ implies that $\tilde b$ and $\tilde
a$ are almost decreasing.
\end{lemma}
\begin{lemma}\label{lm3.6}
Let $E\subseteq\mathbb R^{+},$
$\tilde\tau=\{\tau_n\}_{n\in\mathbb N}\in\tilde E_{0}^{d},$ and let $\{(a_n^{(i)},
b_n^{(i)})\}_{n\in\mathbb N}\in\tilde I_{E},\, i=1,2.$ If $$\tilde
b^{1}\asymp\tilde\tau\asymp\tilde b^{2}$$ where $\tilde
b^{i}=\{b_{n}^{(i)}\}_{n\in\mathbb N}, \, i=1,2,$ then there is $N_0 \in\mathbb N$ such
that
$$(a_{n}^{(1)},b_{n}^{(1)})=(a_{n}^{(2)},b_{n}^{(2)})$$ for every $n\ge N_0.$
\end{lemma}
The next lemma is closely related to the implication $\mathrm{(i)}\Rightarrow
\mathrm{(ii)}$ from Theorem~\ref{ImpTh}.
\begin{lemma}\label{lm3.7}
Let $E\subseteq\mathbb R^{+}$ be strongly porous on the right at 0 and let $0\in ac E.$ If for every $\tilde\tau=\{\tau_n\}_{n\in\mathbb N}\in\tilde E_{0}^{d}$ there is
$\{(a_n, b_n)\}_{n\in\mathbb N}\in\tilde I_{E}$ such that $\{\tau_n\}_{n\in\mathbb N}$
$\asymp\{b_n\}_{n\in\mathbb N},$ then there is an universal $\tilde L=\{(l_n, m_n)\}_{n\in\mathbb N}$ $\in\tilde
I_{E}^{sd}$ with \begin{equation}\label{l3.14} M(\tilde L)<\infty.\end{equation}
\end{lemma}
The following proof is a modification of the corresponding part of the proof of
Theorem~\ref{ImpTh}.
\newline \emph{Proof of Lemma~\ref{lm3.7}.} Suppose that for every $\tilde\tau=\{\tau_n\}_{n\in\mathbb N}\in\tilde
E_0^{d}$ there is $\{(a_n, b_n)\}_{n\in\mathbb N}$ $\in\tilde I_{E}$ such that $\tilde
\tau\asymp\tilde b=\{b_n\}_{n\in\mathbb N}.$ In the proof of Theorem~\ref{ImpTh} we have
found a sequence $\tilde u=\{u_n\}_{n\in\mathbb N}\in\tilde E_0^{d}$ such that for every
$\tilde\tau\in\tilde E_0^{d}$ there is an almost  increasing function $f:\mathbb
N\rightarrow \mathbb N$ satisfying the relation
\begin{equation}\label{l3.15}
\{\tau_k\}_{k\in\mathbb N}\asymp\{u_{f(k)}\}_{k\in\mathbb N}.
\end{equation}
By the supposition there is $\{(a_n, b_n)\}_{n\in\mathbb N}\in\tilde I_{E}$ such that
\begin{equation}\label{l3.15*}
\tilde u\asymp\tilde b.
\end{equation}
Since $\tilde u\in\tilde E_0^{d},$ Lemma~\ref{lm3.5} implies that $\tilde b$ and $\tilde
u$ are almost decreasing. Consequently $\tilde A:=\{(a_n, b_n)\}_{n\in\mathbb
N}\in\tilde I_{E}^{d}.$ We shall show that $\tilde A$ is universal. Let $\tilde
L=\{(l_n, m_n)\}_{n\in\mathbb N}$ be an arbitrary element of $\tilde I_{E}^{d}.$ Using
Definition~\ref{univ*} we see that $\tilde A$ is universal if and only if there are $N_1
\in \mathbb N$ and $f: \mathbb N_{N_1}\rightarrow \mathbb N$ such that
\begin{equation}\label{l3.16}
m_n = b_{f(n)}
\end{equation} for $n\in \mathbb N_{N_1}.$ It is easy to show that there is $\tilde\tau=\{\tau_n\}_{n\in\mathbb N}\in\tilde
E_0^{d}$ such that
\begin{equation}\label{l3.17}
\lim_{n\to\infty}\frac{\tau_n}{m_n}=1.
\end{equation} The last limit relation implies that $\{m_n\}_{n\in\mathbb N}=\tilde m\asymp\tilde\tau=\{\tau_n\}_{n\in\mathbb
N}.$ This equivalence, \eqref{l3.15} and \eqref{l3.15*} give us
$$\{m_k\}_{k\in\mathbb N}\asymp\{b_{f(k)}\}_{k\in\mathbb N}.$$ It is clear that $\{(a_{f(k)},b_{f(k)})\}_{k\in\mathbb N}\in\tilde
I_{E}^{d}.$ Consequently, by Lemma~\ref{lm3.6}, there is $N_0 \in \mathbb N$ such that
$$(l_k, m_k)=(a_{f(k)},b_{f(k)})$$ for all $k\ge N_0.$ Equality \eqref{l3.16} follows
for sufficiently large $n.$ Hence $\tilde A\in\tilde I_{E}^{d}$ is universal. Using
Lemma~\ref{lem2.14} we can assume that $\{a_n\}_{n\in\mathbb N}$ and
$\{b_n\}_{n\in\mathbb N}$ are strictly decreasing. To complete the proof it suffices to
show that $M(\tilde A)<\infty.$ As in the proof of Lemma~\ref{Lem2.15} we may consider
the closed intervals $[b_{n+1}, a_n], \, n=1,2,...,$ that together with the half-open
interval $[b_1, \infty)$  form a disjoint cover of the set $E\setminus \{0\},$
$$E\setminus \{0\}\subseteq[b_1, \infty)\cup\left(\bigcup_{n\in\mathbb N}[b_{n+1}, a_n]\right).$$
We can find a sequence $\tilde\tau=\{\tau_n\}_{n\in\mathbb N}\in\tilde E_0^{d}$ such
that
\begin{equation}\label{l3.18}
\lim_{n\to\infty}\frac{\tau_n}{a_n}=1 \quad \mbox{and} \quad \tau_n \in [b_{n+1}, a_n]
\end{equation} for every $n\in\mathbb N.$ Reasoning as in the proof of equality
\eqref{L29} we can see that $$\{\tau_n\}_{n\in\mathbb N}\asymp\{b_{n+1}\}_{n\in\mathbb
N},$$ i.e.,  there are positive constants $c_1, c_2$ such that $$c_1 b_{n+1}\le
\tau_{n}\le c_2 b_{n+1}.$$ The last inequality and \eqref{l3.18} imply $$\infty> c_2 \ge
\limsup_{n\to\infty}\frac{\tau_n}{b_{n+1}}=\limsup_{n\to\infty}\frac{\tau_n}{a_n}\frac{a_n}{b_{n+1}}=\limsup_{n\to\infty}\frac{a_n}{b_{n+1}}=M(\tilde
A),$$ and so the lemma is proved. $\qquad \qquad \qquad \qquad \qquad \qquad \qquad
\qquad  \qquad \qquad \qquad \qquad \, \,\Box$

\vspace{5mm}
Now we can simply finish the proof of Theorem~\ref{th3.1}.

\bigskip

\noindent \emph{Proof of Theorem~\ref{th3.1}.} \emph{The sufficiency.} Suppose for every
$\tilde\tau=\{\tau_n\}_{n\in\mathbb N}\in\tilde E_0^{d}$ there is $\tilde
h=\{h_n\}_{n\in\mathbb N}\in\tilde H(E)$ such that $\tilde \tau\asymp\tilde h.$ Then, by
Lemma~\ref{lm3.2}, for every $\tilde\tau \in \tilde E_0^{d}$ there is $\{(a_n,
b_n)\}_{n\in\mathbb N}\in\tilde I_{E}$ such that $\tilde\tau\asymp\tilde b.$
Consequently, by Lemma~\ref{lm3.7}, the preordered set $(\tilde I_{E}^{d},\preceq)$ has
an universal element $\tilde L\in\tilde I_{E}^{sd}$ satisfying the inequality $M(\tilde
L)<\infty.$ By Theorem~\ref{ImpTh} $E$ is a \textbf{\emph{CSP}} - set.
$\qquad \qquad \qquad \qquad \qquad \qquad \qquad \quad \quad \qquad \qquad \qquad
  \Box$

\bigskip

  Let $A$ and $B$ be subsets of $\mathbb R^{+}.$ We shall write $A\sqsubseteq B$ if there is
  $t=t(A,B)>0$ such that $$A\cap(0,t)\subseteq B\cap (0,t).$$ The next theorem gives
  a constructive description of the  \textbf{\emph{CSP}} - sets.
  \begin{theorem}\label{descr}
Let $E\subseteq\mathbb R^{+}.$ Then $E$ is a \textbf{\emph{CSP}} - set
if and only if there are $q>1$ and a strictly decreasing
sequence $\{x_n\}_{n\in\mathbb N}, x_n>0$ for $n\in\mathbb N,$ such that
\begin{equation}\label{l3.20}\lim_{n\to\infty}\frac{x_{n+1}}{x_n}=0\end{equation} and
\begin{equation}\label{l3.21}E\sqsubseteq W(q)\end{equation} where \begin{equation}\label{l3.22}W(q):=\bigcup_{n\in\mathbb N}(q^{-1}x_n , q x_n).\end{equation}
  \end{theorem}
\begin{proof}
The theorem is trivial if $0\not\in ac E.$ Let us consider the case when $0\in ac E.$
Suppose that there are $q>1$ and a sequence $\{x_n\}_{n\in\mathbb N}$ of positive real
numbers such that \eqref{l3.20} and \eqref{l3.21} hold. Let $N_1$ and $N_2$ be natural
numbers such that
\begin{equation}\label{l3.23}(q^{-1}x_{n+1}, qx_{n+1})\cap(q^{-1}x_{n}, qx_{n})=\varnothing\end{equation}
for $n\ge N_1$
\begin{equation}\label{l3.24}E\cap(0, t)\subseteq W(q)\cap(0, t)\end{equation}
for $t\le x_{N_2}.$ Then we have $$(qx_{n+1}, q^{-1}x_{n})\subseteq Ext
E$$ for $n\ge N_1 \vee N_2$ and write, in this case, $(l_n, m_n)$ for the unique
connected component of $Ext E$ satisfying the inclusion
\begin{equation}\label{l3.25}(l_n, m_n)\supseteq (qx_{n+1}, q^{-1}x_{n}). \end{equation}
Let $(l_n, m_n):=(l_{N_1 \vee N_2}, m_{N_1 \vee N_2})$ for $n< N_1 \vee N_2.$ We claim
that $\tilde L=\{(l_n, m_n)\}_{n\in\mathbb N}$ is universal. Indeed, \eqref{l3.25} imply that
$$\liminf_{n\to\infty}\frac{m_n}{l_n}\ge q^{-2} \liminf_{n\to\infty}\frac{x_{n+1}}{x_{n}}=\infty.$$
Thus $$\lim_{n\to\infty}\frac{m_n}{l_n}=\infty,$$ so that $\tilde L$ belongs to $\tilde
I_{E}^{d}.$ Let $\tilde A=\{(a_j, b_j)\}_{j\in\mathbb N}$ be an arbitrary element of
$\tilde I_{E}^{d}.$ There is $N_3 \in\mathbb N$ such that
\begin{equation}\label{l3.26}\frac{b_j}{a_j}>q^{2} \end{equation} and $b_j < (x_{N_1} \vee
 x_{N_2})$ for $j\ge N_3.$ Let $j\ge N_3.$ The interval $(a_j, b_j)$ is a connected
 component of $Ext E.$ Consequently, there is $n\ge (N_1 \vee N_2)$ such that either
\begin{equation}\label{l3.27}(a_j, b_j)\supseteq(qx_{n+1}, q^{-1}x_n) \end{equation} or
\begin{equation}\label{l3.28}(a_j, b_j)\subseteq(q^{-1}x_n, x_n). \end{equation}
Inclusion \eqref{l3.28} implies $$\frac{b_j}{a_j}\le\frac{qx_n}{q^{-1}x_n}=q^{2},$$
contrary to \eqref{l3.26}. Hence \eqref{l3.27} holds. Since for every nonvoid interval
$(s,t)\subseteq Ext E$ there is a unique connected component $(a,b)\supseteq (s,t),$
inclusions \eqref{l3.25} and \eqref{l3.27} imply the equality $(l_n, m_n)=(a_j,
b_j).$ Hence $\tilde L\succeq\tilde A$ for every $\tilde A\in\tilde I_{E}^{d}.$ Thus
$\tilde L$ is an universal element of $(\tilde I_{E}^{d}, \preceq).$

In accordance with Theorem~\ref{ImpTh} to prove that $E$ is a \textbf{\emph{CSP}} - set it is sufficient to show \begin{equation}\label{l3.29}M(\tilde
L)=\limsup_{n\to\infty}\frac{l_n}{m_{n+1}}<\infty.
\end{equation}

\begin{figure}[h]
\centerline{\includegraphics[scale=0.7]{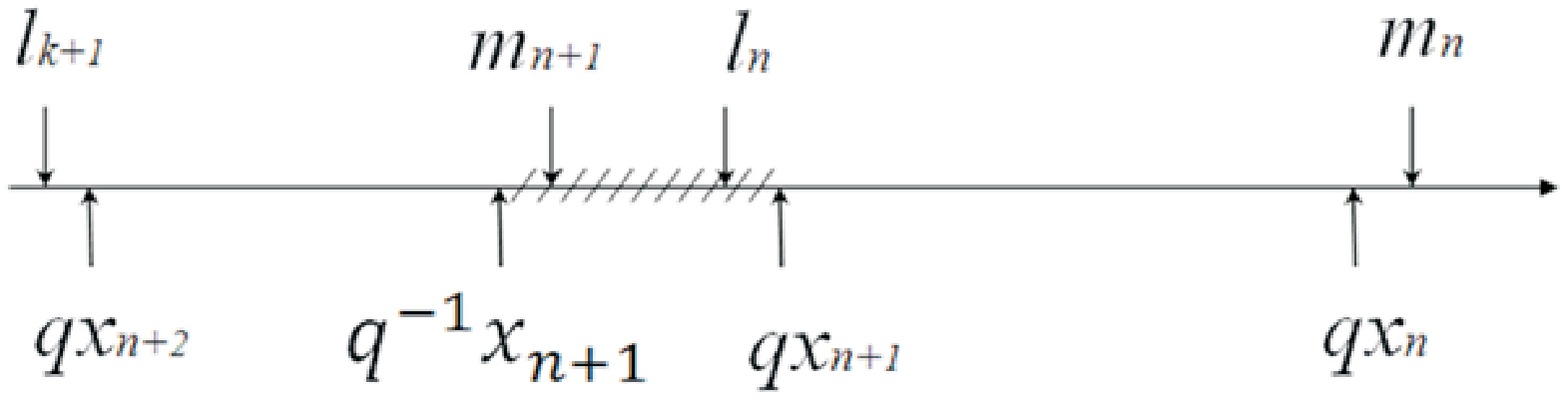}}%
\centerline{\emph{Fig. 3}}
\end{figure}

Since, for sufficiently large $n,$ $(l_n, m_n)\supseteq (qx_{n+1}, q^{-1}x_{n})$ and
$(l_{n+1}, m_{n+1})\supseteq (qx_{n+2}, q^{-1}x_{n+1})$ and $l_{n+1}<m_{n+1}<l_{n}<m_{n}$
and $qx_{n+2}<q^{-1}x_{n+1}<qx_{n+1}<q^{-1}x_{n}$ (see Fig. 3), we have $$m_{n+1},
l_{n}\in [q^{-1}x_{n+1}, qx_{n+1}].$$ Consequently the inequality
$$\frac{l_n}{m_{n+1}}\le \frac{qx_{n+1}}{q^{-1}x_{n+1}}=q^{2}$$ holds for sufficiently
large $n.$ Inequality \eqref{l3.29} follows.

Now assume that $E$ is a \textbf{\emph{CSP}} - set. Let $\{(l_n, m_n)\}_{n\in\mathbb
N}\in\tilde I_{E}^{sd}$ be universal. Without
loss of generality, we may suppose that the sequence $\{l_n\}_{n\in\mathbb N}$ is strictly decreasing. Define $\{x_n\}_{n\in\mathbb
N}:=\{m_n\}_{n\in\mathbb N}.$ Using the inequality $m_{n+1}\le l_n$ we obtain, from the
definition of $\tilde I_{E}^{d},$ that $$\limsup_{n\to\infty}\frac{x_{n+1}}{x_n}\le
\limsup_{n\to\infty}\frac{l_n}{m_n}=0.$$ Thus
$$\lim_{n\to\infty}\frac{x_{n+1}}{x_n}=0.$$
To complete the proof it is sufficient to show that there is $q>1$ such that
\eqref{l3.21} holds. As in the proof of Lemma~\ref{Lem2.15}, one can easily note that
\begin{equation}\label{l3.30}E\setminus \{0\}\subseteq [m_1, \infty) \cup \left (\bigcup_{n\in\mathbb N}[m_{n+1},
l_n]\right).
\end{equation}
By formulas \eqref{L13} and \eqref{L14} we have $$M(\tilde
L)=\limsup_{n\to\infty}\frac{l_n}{m_{n+1}}<\infty.$$ Let $q\in (M(\tilde L), \infty).$
Then there is $N_4 \in \mathbb N$ such that $\frac{l_n}{m_{n+1}}<q$ for $n\ge N_4.$
It is clear that $q>1.$ Consequently the inequalities
$q^{-1}m_{n+1}<m_{n+1}\le l_{n}<qm_{n+1}$ hold for $n\ge N_4.$ These inequalities yield the inclusion $[m_{n+1}, l_n]\subseteq(q^{-1}m_{n+1}, qm_{n+1}).$ The
last inclusion and \eqref{l3.30} imply $$E\cap (0,t)\subseteq \left
(\bigcup_{n\in\mathbb N}(q^{-1}m_{n}, qm_n)\right)\cap (0,t)$$ for every $t\in (0,
m_{N_{4}+1}).$ Relation \eqref{l3.21} follows.
\end{proof}
  In the case of the closed sets $E$ we may modify Theorem~\ref{descr} by the
  following way.
  \begin{theorem}\label{descr*}
Let $E\subseteq\mathbb R^{+}$ be closed and let $0\in ac E.$ Then $E$ is
a \textbf{\emph{CSP}} - set if and only if there are $q>1$ and a strictly decreasing
sequence of numbers $x_n >0, \, n\in\mathbb N,$ such that
$\mathop{\lim}\limits_{n\to\infty}\frac{x_{n+1}}{x_n}=0$ and$$W(1)\sqsubseteq
E\sqsubseteq W(q)$$ where
$$W(a)=\left(\bigcup_{n\in\mathbb N}[x_{n}, ax_{n}]\right), \, a\in [1, \infty).$$
  \end{theorem}
  The last theorem shows that examples \ref{ex1} and \ref{ex2} give us, in a sense,``the
  extremal elements'' among the closed \textbf{\emph{CSP}} - sets with the accumulation point 0. The proof of Theorem~\ref{descr*} is similar to the proof of Theorem~\ref{descr}, so
  we omit it here.


\medskip
\noindent
{\bf Viktoriia Bilet} \hspace{5cm} {\bf Oleksiy Dovgoshey}\\
IAMM of NASU, Donetsk \hspace{3.4cm} IAMM of NASU, Donetsk \\
E-mail: biletvictoriya@mail.ru \hspace{2.6 cm} E-mail: aleksdov@mail.ru  \\

\end{document}